\documentclass{amsart}
\usepackage{latexsym,amsxtra,amscd,ifthen,amsmath,color}
\usepackage{amsfonts}
\usepackage{verbatim}
\usepackage{amsmath}
\usepackage{amsthm}
\usepackage{amssymb}
\usepackage[notcite,notref,final]{showkeys}

\numberwithin{equation}{section}

\theoremstyle{plain}
\newtheorem{theorem}{Theorem}[section]
\newtheorem{lemma}[theorem]{Lemma}

\newtheorem{proposition}[theorem]{Proposition}
\newtheorem{hypothesis}[theorem]{Hypothesis}
\newtheorem{corollary}[theorem]{Corollary}

\newcommand{\inth}{\textstyle \int}
\theoremstyle{definition}
\newtheorem{definition}[theorem]{Definition}
\newtheorem{notation}[theorem]{Notation}
\newtheorem{example}[theorem]{Example}

\newtheorem{remark}[theorem]{Remark}
\newtheorem{question}[theorem]{Question}

\makeatletter              
\let\c@equation\c@theorem  
\makeatother

\DeclareMathOperator{\Rees}{Rees}

\DeclareMathOperator{\hdet}{hdet}

\DeclareMathOperator{\gldim}{gldim}

\DeclareMathOperator{\Spec}{Spec}

\DeclareMathOperator{\gr}{gr}

\DeclareMathOperator{\tr}{tr}

\DeclareMathOperator{\Aut}{Aut}

\DeclareMathOperator{\injdim}{injdim}

\DeclareMathOperator{\End}{End}

\newcommand{\ch}{\operatorname{char}}

\newcommand{\cohdet}{\mathrm{hcodet}}

\begin{document}

\title{Hopf actions on filtered regular algebras}

\author{K. Chan, C. Walton, Y.H. Wang and J.J. Zhang}

\address{Chan: Department of Mathematics, Box 354350,
University of Washington, Seattle, Washington 98195,
USA}

\email{kenhchan@math.washington.edu}

\address{Walton: Department of Mathematics, Massachusetts
Institute of Technology, Cambridge, Massachusetts 02139,
USA}

\email{notlaw@math.mit.edu}

\address{Wang: Department of Applied Mathematics,
Shanghai University of Finance and Economics, Shanghai 200433, China}

\email{yhw@mail.shufe.edu.cn}

\address{Zhang: Department of Mathematics, Box 354350,
University of Washington, Seattle, Washington 98195,
USA}

\email{zhang@math.washington.edu}




\maketitle

\begin{abstract}
We study finite dimensional Hopf algebra actions on so-called filtered
Artin-Schelter regular algebras of dimension $n$, particularly on those of dimension 2. The first Weyl algebra is an example of such on algebra with $n=2$, for instance. Results on the Gorenstein condition and on the global dimension of the corresponding fixed subrings are also provided.
\end{abstract}

\maketitle


\setcounter{section}{-1}
\section{Introduction}
\label{xxsec0}

The main motivation for this paper (as well as for \cite{CWZ,CKWZ})
is to classify all finite dimensional Hopf algebras which act on a
given algebra $R$. By understanding the Hopf algebras $H$ which act
on $R$, we can further study other structures related to $R$, such
as the fixed ring $R^H$ and the smash product $R\#H$. The
prototype of this problem is classical: the classification of finite
subgroups $G$ of $SL_2(\mathbb{C})$ (that act faithfully on the
polynomial ring $\mathbb{C}[u,v]$) prompted the connection between
the McKay quiver of $G$ and the geometric features of the plane
quotient singularity $\Spec(\mathbb{C}[u,v]^G)$. In our
setting, the algebra $R$ is allowed to be
noncommutative and the Hopf algebras are allowed to be
noncocommutative. More precisely, we study finite dimensional Hopf
algebra actions on {\it filtered Artin-Schelter regular algebras of dimension $d$}. These
are filtered algebras whose associated graded algebras are Artin-Schelter regular algebras of global dimension $d$. Our emphasis
will be on the case of dimension $2$.

Here, we assume that the base field $k$ is algebraically
closed of characteristic zero, unless otherwise stated. Examples of
filtered Artin-Schelter regular algebras of dimension 2 include the first Weyl algebra
$A_1(k) = k\langle u,v \rangle/(vu-uv-1)$, quantum Weyl algebras
$k\langle u,v\rangle/(vu-quv-1)$ for some $q\in k^{\times}$, and other deformations of Artin-Schelter regular
algebras of dimension 2.

The invariant theory of $A_1 \left( k \right)$ by finite groups is already
interesting. For example, the fixed subrings of $A_1 \left( k \right)$ by
finite groups actions are completely classified and studied by Alev-Hodges-Velez in \cite{AHV}. Thus, it is
natural to ask if there are any
non-trivial finite dimensional Hopf algebra ($H$-) actions on the
first Weyl algebra. By a ``non-trivial" $H$-action, we mean that $H$
is neither commutative nor cocommutative Hopf algebra (or neither a
dual of a group algebra nor group algebra, respectively).
We give a negative answer to this question in Theorem \ref{thmxx0.1} below.

Recall that a left $H$-module $M$ is called {\it inner-faithful} if
$IM\neq 0$ for any nonzero Hopf ideal $I$ of $H$. Let $N$ be a right
H-comodule with comodule structure map $\rho : N \rightarrow
N\otimes H$, We say that this coaction is inner-faithful if
$\rho(N)\nsubseteq N\otimes H^\prime$ for any proper Hopf subalgebra
$H^\prime\subsetneq H$ \cite[Definition 1.2]{CWZ}. We say that {\it
a Hopf algebra $H$ (co)acts on an algebra $R$} if $R$ is a left
$H$-(co)module algebra. Moreover, if the $H$-(co)module $R$ is
inner-faithful, then we say that $H$ (co)acts on $R$ {\it
inner-faithfully}. 

\begin{theorem}
\label{thmxx0.1}
Let $R$ be a non-PI filtered Artin-Schelter regular algebra of
dimension 2 and let $H$ be a finite dimensional Hopf algebra acting
on $R$ inner-faithfully. If the $H$-action preserves the filtration
of $R$, then $H$ is a group algebra.

In particular, if $R$ is the first Weyl algebra $A_1(k)$, then $H$
is a group algebra.
\end{theorem}

Theorem \ref{thmxx0.1} can be viewed as an extension of
\cite[Theorem 5.10]{CWZ} from the non-PI graded case to the non-PI filtered case. Most results in this work concern actions on non-PI AS regular algebras.
In particular, combining Theorem \ref{thmxx0.1} with \cite[Proposition,
page 84]{AHV}, one classifies all finite-dimensional Hopf algebras
acting inner-faithfully on $A_1(k)$ with respect to the standard
filtration [Corollary \ref{corxx5.8}]. Similarly, all finite
dimensional Hopf algebras actions on the quantum Weyl algebras
$k\langle u,v\rangle/(vu-quv-1)$, for $q$ not a root of unity, are
classified [Corollary \ref{corxx5.9}(a)]. On the other hand, if $R$ is PI filtered AS regular, then there are
many interesting finite dimensional Hopf algebras (which are not
group algebras) which act on $R$; see Examples~\ref{exxx1.4} and~\ref{exxx3.4} for
instance.

Regarding the higher dimensional Weyl algebras, it is natural to ask
the following question.

\begin{question}
\label{quexx0.2} Let $A_n(k)$ be the $n$-th Weyl
algebra and let $H$ be a finite dimensional Hopf algebra acting on
$A_n(k)$ inner-faithfully. Is then $H$ a group algebra?
\end{question}

If we assume that the $H$-action is filtration preserving, then the
answer is yes if $n=1$ [Theorem \ref{thmxx0.1}] or if $H$ is pointed
[Theorem \ref{thmxx0.3}].

\begin{theorem}
\label{thmxx0.3} Let $H$ be a finite dimensional Hopf algebra acting
on the $n$-th Weyl algebra $A_n(k)$ inner-faithfully which preserves
the standard filtration of $A_n(k)$. Then $H$ is semisimple. If, in
addition, $H$ is pointed, then $H$ is a group algebra.
\end{theorem}

In the setting of $H$-actions on graded algebras, we have the following
result of Kirkman-Kuzmanovich-Zhang. Suppose that $H$ is a semisimple finite dimensional Hopf algebra and
$R$ is an AS regular algebra. If $H$ acts on $R$ preserving the grading with
trivial homological determinant, then the fixed subring $R^H$ is AS
Gorenstein \cite[Theorem 0.1]{KKZ}. We can obtain a filtered analogue of the above by considering the
induced $H$-action on $\gr_F R$.

\begin{theorem}
\label{thmxx0.4} Let $R$ be a filtered  AS regular algebra of
dimension 2 and let $H$ be a semisimple Hopf algebra acting on $R$.
If the $H$-action is not graded and it preserves the filtration of $R$,
then the fixed subring $R^H$ is filtered AS Gorenstein.
\end{theorem}

Note that if $R$ is a graded $H$-module algebra, the fixed ring $R^H$ can
fail to be AS Gorenstein. This is well known if $R=k[u,v]$ by the existence of
non-Gorenstein quotient singularities. See also \cite[Proposition 6.5(2)]{JiZ} for
a noncommutative example.

On the other hand, we also study the regularity of fixed rings when $H$ is not necessarily semisimple. In the graded
case, \cite[Proposition 0.7]{CKWZ} states that if $R$ is graded AS
regular of global dimension 2, and the $H$-action on $R$ has trivial
homological determinant, then $R^H$ is \underline{never} AS regular provided that $H \neq k$. Therefore,
the following result is quite surprising.

\begin{theorem}
\label{thmxx0.5} Let $R$ be a non-PI filtered AS regular algebra of
dimension 2 and let $H$ be a finite dimensional Hopf algebra acting
inner-faithfully on $R$. If the $H$-action on $R$ is not graded and preserves the
filtration on $R$, then the fixed subring $R^H$ has global dimension
$1$ or $2$.
\end{theorem}

This result is well-known when $R$ is the first Weyl algebra
$A_1(k)$ as the corresponding fixed subrings are all hereditary \cite{AHV}. On the other hand, it would be interesting to prove versions of Theorems
\ref{thmxx0.4} and \ref{thmxx0.5} in the higher dimensional case.
\begin{remark}
Suppose that $H$ is semisimple and $A_n(k) \# H$ is simple. 
Then by \cite[Corollary~4.5.5]{Mo}, we have a  Morita 
equivalence between $A_n(k)^H$
and $A_n(k)\# H$. As a consequence, gldim $A_n(k)^H$  
= gldim $A_n(k) \# H$  = gldim $A_n(k)$. \end{remark}

Moreover, observe that since $R$ is non-PI in Theorem 
\ref{thmxx0.5}, there are no
non-trivial Hopf actions on $R$ by Theorem \ref{thmxx0.1}.

The paper is organized as follows. We define basic terminology and
we also discuss certain properties of Hopf actions on filtered
algebras in Section \ref{xxsec1}. In Section \ref{xxsec2}, we make
several initial computations on the structure of a Hopf algebra $H$
and a filtered AS regular algebra $R$ of dimension 2, particularly
when the $H$-(co)action on $R$ is so-called {\it proper}. We provide
preliminary results about the fixed subring $R^H$ in Section
\ref{xxsec3}, and we also prove Theorem \ref{thmxx0.4} here. The
proof of Theorem \ref{thmxx0.3} is presented in Section
\ref{xxsec4}. In Section 5, we prove Theorem \ref{thmxx0.1} and we
use it to classify Hopf actions on non-PI filtered AS regular
algebras of dimension $2$. Finally in Section \ref{xxsec6}, we prove
Theorem \ref{thmxx0.5} with use of Galois extensions.

\section{Definitions}
\label{xxsec1}

In this work, we study Hopf actions on filtered algebras whose
associated graded algebras are Artin-Schelter (AS) regular algebras
of global dimension 2. We refer to \cite[Definition 1.1]{CWZ} for
the definition of AS regular algebras in general, but in global
dimension 2 (and generated in degree 1), we have that such algebras
are isomorphic to either:
\[
\begin{array}{lll}
\text{(i)}  &k_J[u,v]:= k\langle u,v\rangle/(vu-uv-u^2)
& \text{({\it the Jordan plane})}; \text{~or} \\
\text{(ii)}  &k_q[u,v]:=k\langle u,v\rangle/(vu-quv) \text{~for some
$q\in k^{\times}$} & \text{({\it the skew polynomial ring})}.
\end{array}
\]
 For filtered algebras, we have the following definition.

\begin{definition}
\label{defxx1.1} An algebra $R$ is called {\it filtered AS regular
of dimension $d$} (respectively, {\it filtered AS Gorenstein}) if
the following conditions hold:
\begin{enumerate}
\item
$R$ is generated by a finite dimensional subspace $U$ with $1\not\in U$, and
\item
for $F_n=(k1+U)^n$,
the associated graded ring
$$\gr_F R:=\bigoplus_{n\geq 0} F_n/F_{n-1}$$
is an AS regular algebra of global dimension $d$ (respectively, is
an AS Gorenstein algebra).
\end{enumerate}
\end{definition}

We define below {\it actions} of Hopf algebras on filtered
regular algebras, and in particular, actions that are so-called {\it proper}.

\begin{notation}
\label{notxx1.2} We denote by $R$ a filtered AS regular algebra of
dimension $d$. If $d=2$, then the generating vector space $U$ has dimension 2 and we use $\{u,v\}$ for its
basis. Unless otherwise stated, $H$ and $K$ are finite
dimensional Hopf algebras. Here, the Hopf algebras have Hopf
structure denoted by the standard notation $H = (H, m, \Delta, u,
\epsilon, S)$. Moreover, for the following
definition, we denote the left $H$-action on $R$ by $\cdot: H
\otimes R \rightarrow R$, and right $K$-coaction on $R$ by $\rho: R
\rightarrow R \otimes K$.
\end{notation}

Since $R$ is a filtered algebra, we require Hopf actions
(or Hopf coactions) to preserve the given filtration of $R$. We give only
the following definitions for Hopf actions, but similar definitions
can be made for Hopf coactions with the obvious changes.

\begin{definition}
\label{defxx1.3} We say a Hopf algebra $H$ {\it acts on a filtered
algebra} $R$ if
\begin{enumerate}
\item
$R$ is a left $H$-module algebra, and
\item
$k1+U$ is a left $H$-module.
\end{enumerate}
We say the $H$-action on $R$ is {\it proper} if, further,
\begin{enumerate}
\item[(c)]
there is a choice of $U$ as in Definition \ref{defxx1.1}
such that $U$ is a left $H$-module.
\end{enumerate}
\end{definition}

Let us provide an example of a non-proper  Hopf coaction on a PI filtered AS regular algebra.

\begin{example}
\label{exxx1.4} Let $R$ be the quantum Weyl algebra $k\langle
u,v\rangle/(vu+uv-1)$ (where the parameter $q$ equals $-1$), and let
$K$ be the Sweedler's non-semisimple 4-dimensional Hopf algebra
$k\langle g, f\rangle/(fg+gf, g^2-1, f^2)$. Here, $g$ is grouplike
and $f$ is $(1,g)$-primitive, to say,
$$\Delta(g) = g \otimes g \quad \quad \text{and}
\quad \quad \Delta(f)=1\otimes f+f\otimes g.$$
Moreover,
$\epsilon(g) =1$, $\epsilon(f) = 0$, $S(g) = g$, and $S(f) = -fg$.
Define a $K$-coaction on $R$ by
$$\begin{aligned}
\rho(u)=u\otimes g \quad \quad \text{and} \quad \quad
\rho(v)=v\otimes g + 1\otimes f.
\end{aligned}
$$
Then, $K$ coacts on $R$ inner-faithfully. We show that the induced $K^{\ast}=H$-action
is not proper, hence the $K$-coaction above is not proper.

Let $e_1, e_g, e_f, e_{g f}$ denote the dual basis of $H$. Then
$\gamma := e_1 - e_g$ and $\delta := e_f - e_{g f}$ generate $H$ as
an algebra. Moreover $\gamma \cdot 1=1$, $\gamma \cdot u=-u$,
$\gamma \cdot v=-v $ and $\delta \cdot 1 =\delta \cdot u = 0$,
$\delta \cdot  v  = 1$. By linear algebra, there is no
$2$-dimensional $\delta$-invariant subspace of $F_1=k1\oplus kv
\oplus ku$ not containing $1$. So the $H$-action is not proper.
\end{example}

We require (co)actions that are `faithful' in our setting. We refer
to \cite[Section~1]{CWZ} for a discussion of inner-faithful Hopf
(co)actions, and we repeat some of these results here.

\begin{lemma}
\label{lemxx1.5} Let $H$ be a Hopf algebra that acts on a filtered
algebra $R$.
\begin{enumerate}
\item
If $H$ is semisimple, then every $H$-action on $R$ is proper.
\item
\cite[Lemma 1.3(c)]{CWZ}
The $H$-action on $R$ is inner-faithful if
and only if the $H$-module $k1+U$ is inner-faithful.
\item
If the $H$-action is proper, the $H$-module $k1+U$ is inner-faithful if and
only if the $H$-module $U$ is inner-faithful.
\item \cite[Lemma 1.3(a)]{CWZ}
If $H$ is finite dimensional with Hopf algebra dual $H^\circ$, then the
$H$-action is inner-faithful if and only if the $H^\circ$-coaction is
inner-faithful.
\end{enumerate}
\end{lemma}

\begin{proof} (a) Since $H$ is semisimple, $k1+U=k1\oplus U'$ which is
a direct sum of left $H$-modules $k1$ and $U'$ where $U'$ is a finite
dimensional generating subspace of $R$. Replacing $U$ by $U'$ gives
the assertion.

(b,c,d) These are straightforward.
\end{proof}

We also work with Hopf actions that are {\it not graded} as in the
following definitions.

\begin{definition}
\label{defxx1.6} Let $R$ be a filtered AS regular of dimension $d$
and let $H$ act on $R$.
\begin{enumerate}
\item
We say that {\it $R$ is not graded} if $R$ is not isomorphic to
$\gr_F R$ for any choice of $U$ in Definition~\ref{defxx1.1}.
\item
We say that the {\it $H$-action is not graded}, if for any choice of
$U$ in Definition~\ref{defxx1.3}, $R$ is not isomorphic to $\gr_F R$
as left $H$-module algebras.
\end{enumerate}
\end{definition}

Next, we recall the definition of the {\it homological determinant} of an
$H$-action on a graded algebra $A$.

\begin{definition}
\label{defxx1.7}\cite[Definitions 3.3 and 6.2]{KKZ} Let $A$ be a noetherian
connected graded AS Gorenstein algebra and let $H$ be a finite
dimensional Hopf algebra acting on $A$ that preserves the grading of
$A$.  Let ${\mathfrak e}$ denote the lowest degree nonzero
homogeneous component of the $d$-th local cohomology
$H^d_{{\mathfrak m}_{A}}(A)^{\ast}$, where $d =\injdim(A) < \infty$.
Then there is an algebra homomorphism $\eta: H \rightarrow k$ such
that
$${\mathfrak e}\cdot h =\eta(h){\mathfrak e}$$
for all $h \in H$.
\begin{enumerate}
\item[(1)] The composite map $\eta
\circ S: H \rightarrow k$ is called the {\it homological
determinant} of the $H$-action on $A$, denoted by $\hdet_H A$.
\item[(2)] We
say that $\hdet_H A$ is {\it trivial} if $\hdet_H A= \epsilon$, where
$\epsilon$ is the counit of $H$.
\end{enumerate}

Dually, if a finite dimensional Hopf algebra $K$ coacts on $A$ from
the right, then $K$ coacts on $k{\mathfrak e}$ and
$$\rho({\mathfrak e})={\mathfrak e}\otimes \sf{D}^{-1}$$
for some grouplike element $\sf{D}$ in $K$.
\begin{enumerate}
\item[(3)] The {\it homological
codeterminant} of the $K$-coaction on $A$ is defined to be
$\cohdet_K A =\sf{D}$.
\item[(4)] We say that $\cohdet_K $
is {\it trivial} if $\cohdet_K =1_K$.
\end{enumerate}
\end{definition}

\section{Initial analysis}
\label{xxsec2}

In this section, we compute the structure of the pair $(H, R)$ for
$R$ a filtered AS regular algebra of dimension 2 and $H$ a finite
dimensional Hopf algebra so that $H$ acts on $R$ under various
conditions. We do this particularly when the $H$-action on $R$ is
proper [Lemma \ref{lemxx2.7}, Remark \ref{rmk:HinSec2}, Corollary
\ref{corxx2.8}]. Moreover, we end this section by showing that if
the associated graded algebra of a certain filtered AS regular
algebra $R$ of dimension 2 is PI, then so is $R$ [Lemma \ref{lemxx2.9}].

\begin{notation}
\label{notxx2.1} Let $H$ be a Hopf algebra. Denote by $G:=G(H)$ the
set of grouplike elements in $H$, and let $kG$ be the corresponding
group algebra, which is a Hopf subalgebra of $H$. For $g\in G$,
denote by $\eta_g(h)$ the element $g^{-1}h g$ for any $h\in H$. For
a polynomial $p(t)=\sum_{s=0}^n a_s t^s$, denote by
$(p\circ\eta_g)(h)$ the element $\sum_{s=0}^n a_s \eta_g^s(h)$. Let
$\mathbb{U}_n$ be the set of primitive $n$-th roots of unity for
$n\geq 2$, and put $\mathbb{U}:=\bigcup_{n\geq 2} \mathbb{U}_n$.
\end{notation}

Consider the following preliminary results.

\begin{lemma} \label{lem:1dimlcoact} Let $H$ be a finite dimensional Hopf algebra and
$K:=H^{\circ}$.
If $T$ is a 1-dimensional right $K$-comodule, then $T \cong kg$ for
some grouplike element $g \in G(K)$.
\end{lemma}

\begin{proof} Take a nonzero basis element $t$ of $T$. Now
$\rho(t) = t \otimes g$, and by coassociativity,
$$t \otimes \Delta(g) = (1 \otimes \Delta) \circ \rho(t) =
(\rho \otimes 1) \circ \rho(t) = t \otimes g \otimes g.$$
Hence, $\Delta(g) = g \otimes g$.
\end{proof}

\begin{lemma} \label{lemxx2.2}
Let $H$ be a finite dimensional Hopf algebra, and $g, h \in G$.
Suppose that $f$ is a $(1,g)$-primitive element not in $kG$. Then, the following statements hold.
\begin{enumerate}
\item
We have that $g\neq 1$ and there is no nonzero primitive
element in $H$.
\item
If $\eta_h \left( f \right) - q f \in k G$ for some $q \in k$, then
$q \in \mathbb{U} \cup \left\{ 1 \right\}$.
\item
If $\eta_g \left( f \right) - q f \in k G$ for some $q \in k$, then
$q \in \mathbb{U}$.
\end{enumerate}
\end{lemma}

\begin{proof}
(a)  Suppose that $g=1$. Then $f$ is a primitive element and the Hopf subalgebra generated by $f$ is infinite dimensional.  This yields a contradiction. Therefore, $g\neq 1$.

(b) First, by induction we have that $\eta_h^s \left( f \right) -
q^s f \in k G$ for every $s \geqslant  1$.
Now $h$ has finite order as $H$ is finite dimensional. So $f -
q^m f  = (1-q^m)f \in k G$ where $m$ is the order of $h$. Since $f
\not\in k G$, we obtain that $q^m = 1$.

(c) By part (b), it suffices to show that $q\neq 1$. Suppose that
$q=1$ and that $\eta_g(f) - f \in kG$. Note that $\eta_g(f) -f$ is also
$(1,g)$-primitive, so $\eta_g(f) - f = \alpha(1-g)$ for some $\alpha
\in k$. By induction, one sees that
$$\eta_g^i(f)=f+i\alpha (1-g)$$
for all $i\geqslant 1$. Let $m$ denote the order of $g$, then
$$f = \eta_g^m(f) =f+m\alpha (1-g).$$
Since $g \neq 1$ by part (a) and $m>0$, we have $\alpha =0$. This
implies $[g,f]=0$, so the sub-Hopf algebra $S$ generated by $f,g$ is
commutative. In particular, $S$ is cosemisimple. Since $\ch(k)=0$,
we have $S$ is also semisimple. Now $S$ is generated by a grouplike
element and a skew primitive element, so it is pointed. Therefore
$S$ is a group algebra which contradicts $f\not\in kG$. Hence,
$q\neq 1$.
\end{proof}

We impose the following hypotheses for the next several results.

\begin{hypothesis} \label{hypxx2.3}
Let $R$ be a filtered AS regular algebra of dimension 2, and let
$F=\{F_n\mid n\geqslant 0\}$ be the filtration of $R$ given in
Definition \ref{defxx1.1}. Let $H$ be a finite dimensional Hopf
algebra that acts on $R$ inner-faithfully such that the $H$-action
preserves the filtration of $R$. Moreover, we assume that the
$H$-action on $R$ is proper and not graded.
\end{hypothesis}

Now, we prove several preliminary results that we use throughout this paper.

\begin{lemma}
\label{lemxx2.4} Let $R$ and $H$ be as in Hypothesis \ref{hypxx2.3}.
Then
the relation $r$ of $R$ is of the form
$$vu-quv-\lambda u^2+au+bv+c$$
where $\{u,v\}$ is a suitable basis of $U$ and where $q\in
k^\times,$ $\lambda=0$ or $1$, and $a,b,c\in k$.
\end{lemma}

\begin{proof}
The relation of $\gr_F R$ is of the form $vu-quv-\lambda
u^2$, so the assertion follows.
\end{proof}

\begin{lemma}
\label{lemxx2.5}  Let $R$ and $H$ be as in Hypothesis
\ref{hypxx2.3},  and $r$ be the relation of $R$. If $U$ is a simple
left $H$-module, then $a=b=0$.
\end{lemma}

\begin{proof}
We can view the relation $r$ in Lemma \ref{lemxx2.4} as an
element of $U^{\otimes 2} \oplus U \oplus k \subset k \langle U
\rangle$. Since $k r$ is a 1-dimensional $H$-module, the image of $k
r$ under the projection map $ \pi: U^{\otimes 2} \oplus U \oplus k
\longrightarrow U$ is an $H$-module of dimension at most 1. Since
$U$ is a simple 2-dimensional $H$-module, we have  that $\pi \left(
r \right) = 0$. Hence, $a = b = 0$.
\end{proof}

If $(a,b) = (0,0)$, then the following result discusses the
structure of $R$ and the homological determinant of the $H$-action on $\gr_F R$.

\begin{lemma}
\label{lemxx2.6} Let $R$ and $H$ be as in Hypothesis \ref{hypxx2.3}.
Assume that $H \neq k$ and consider the relation $r$ of $R$ given by
Lemma \ref{lemxx2.4}. If $a=b=0$, then $c\neq 0$ and the
homological determinant of the $H$-action on $\gr_F R$ is trivial.
\end{lemma}

\begin{proof}
Since $H$-action is not graded, if $a = b = 0$, then $c \neq 0$.
Since $k r$ is a $1$-dimensional $H$-module, for each $h \in H$, the
equation
$$ h \cdot \left( v u - q u v - \lambda u^2 + c \right)  =  \phi \left( h
  \right) \left( v u - q u v - \lambda u^2 + c \right)$$
 defines an algebra map $\phi : H \longrightarrow k$.
Since $h \cdot c = \epsilon \left( h \right) c$, we see that $\phi =
\epsilon$. Hence
$$ h \cdot \left( v u - q u v - \lambda u^2 \right)  =  \epsilon \left( h
  \right) \left( v u - q u v - \lambda u^2 \right).$$
By \cite[Theorem 2.1]{CKWZ}, the homological determinant of the
$H$-action on $\gr_F R$ is now trivial as desired.
\end{proof}

On the other hand, if $(a,b) \neq (0,0)$, then we have the following result.

\begin{lemma} \label{lemxx2.7}
Let $R$ and $H$ be as in Hypothesis \ref{hypxx2.3} with $H \neq k$.
Consider the relation $r$ of $R$ as in Lemma \ref{lemxx2.4}. If
$(a,b)\neq (0,0)$ (for any choice of basis $\{u,v\}$), then we have the following statements.
\begin{enumerate}
\item
$U$ is a direct sum of two 1-dimensional $H$-modules: $U \cong T_1
\oplus T_2$.
\item
Here, $T_1 \cong k$ and $T_2 \not \cong k$, or $T_1 \not \cong k$
and $T_2 \cong k$.
\item
$H \cong kC_m$, a cyclic group algebra for $m \geqslant 2$.
\item
The relation $r$ of $R$ is of the form $vu-uv -v$ up to a change of
basis.
\item
The homological determinant of the $H$-action on the associated graded
ring $\gr_F R$ is
non-trivial.  (Equivalently, if $\hdet_H \gr_F R= \epsilon$, then
$(a,b) = (0,0).$)
\end{enumerate}
\end{lemma}

\begin{proof}
(a) We work with the coaction of $K:=H^{\circ}$ instead of the action of
$H$. Suppose to the contrary that $U$ is indecomposable. We will show
that in this case the $K$-coaction $\rho: R \rightarrow R \otimes K$
on $R$ is graded, thus producing a contradiction.

Since $(a,b)\neq (0,0)$, by Lemma \ref{lemxx2.5}, the $K$-comodule
$U$ is not simple. So there is a non-split exact sequence of $K$-comodules
$$0\to T_1\to U\to T_2\to 0$$
where $T_1$ and $T_2$ are 1-dimensional.
Choose a basis $\{u,v\}$ of $U$ such that $u\in T_1$ with $v \in U \setminus T_1$,
then
\begin{equation}\label{E2.7.1}\tag{E2.7.1}
\begin{cases}
\rho(u)=u\otimes g_1, &\\
\rho(v)=v\otimes g_2+u\otimes h&
\end{cases}\end{equation}
We claim that $g_1$, $g_2$ are grouplike and $h$
is $(g_1,g_2)$-primitive. Applying the coassociativity equation
$(\rho \otimes 1)\rho = (1 \otimes \Delta) \rho$ to \eqref{E2.7.1} gives
\begin{eqnarray*}
u \otimes \Delta(g_1) &=&  u \otimes g_1 \otimes g_1 \\
v \otimes \Delta \left( g_2 \right) + u \otimes \Delta \left( h \right) &=& v \otimes g_2 \otimes g_2 + u \otimes h
\otimes g_2 + u \otimes g_1 \otimes h.
\end{eqnarray*}
The claim then follows by comparing coefficients.

With respect to this basis $\{u,v\}$ of $U$, write the relation $r$ of $R$ as
$$r=a_{11}u^2+a_{12}uv+a_{21}vu+a_{22}v^2+au+bv+c.$$
for some scalars $a_{ij}$, $a,b,c \in k$. A simple calculation
shows that
\begin{eqnarray*}
  \rho \left( r \right) & = & u^2 \otimes ( a_{11} g_1^2 + a_{12} g_1 h +
  a_{21} h g_1 + a_{22} h^2) + v^2 \otimes a_{22} g_2^2\\
  &  & + u v \otimes ( a_{12} g_1 g_2 + a_{22} h g_2) + v u
  \otimes ( a_{21} g_2 g_1 + a_{22} g_2 h)\\
  &  & + u \otimes ( a g_1 + b h) + v \otimes b g_2 + 1 \otimes c.
\end{eqnarray*}
Since $\rho \left( r \right) = r \otimes g$ for some grouplike element $g$ by Lemma~\ref{lem:1dimlcoact},
we
have the following equations
\begin{align}
  a_{11} g & \quad=\quad a_{11} g_1^2 + a_{12} g_1 h
    + a_{21} h g_1 + a_{22} h^2,
  \label{E2.7.2}\tag{E2.7.2}\\
  a_{12} g & \quad=\quad a_{12} g_1 g_2 + a_{22} h g_2,
  \label{E2.7.3}\tag{E2.7.3}\\
  a_{21} g & \quad=\quad a_{21} g_2 g_1 + a_{22} g_2 h,
  \label{E2.7.4}\tag{E2.7.4}\\
  a_{22} g & \quad=\quad a_{22} g_2^2,
  \label{E2.7.5}\tag{E2.7.5}\\
  a g & \quad=\quad a g_1 + b h,  \label{E2.7.6}\tag{E2.7.6}\\
  b g & \quad=\quad b g_2, \label{E2.7.7}\tag{E2.7.7}\\
  c g & \quad=\quad c  \label{E2.7.8}\tag{E2.7.8}
\end{align}
If $b \neq 0$, then by (\ref{E2.7.6}), we have $h = a b^{- 1}(g -
g_1) \in k G$, so $K \cong k G$, a contradiction. If $a_{22} \neq
0$, then by (\ref{E2.7.3}) we have that  $h = a_{22}^{-
1} ( a_{12} g g_2^{-1} - a_{12} g_1) \in k G$, which again is a
contradiction. Hence we have $b = a_{22} = 0$. By hypothesis $a \neq
0$, so by (\ref{E2.7.6}) we have $g = g_1$.

If $a_{12} = 0$ or $a_{21} = 0$, then $\gr_F R$ fails to be a domain.
Therefore $a_{12} a_{21} \neq 0$, hence by (\ref{E2.7.3}),
we have that $g = g_1 g_2$. Since $g = g_1$, we conclude $g_2 = 1$.

Now if $c \neq 0$, then by (\ref{E2.7.8}) we have $g = 1$, so $g_1 =
1$. Thus $h$ is a primitive element, which contradicts finite
dimensionality of $K$ [Lemma \ref{lemxx2.2}(a)]. Hence $c = 0$.

Since $a_{21} \neq 0$, without loss of generality, we can take
$a_{21}=1$. Moreover, write $r = vu-quv-\lambda u^2 +au$ with
$a_{11} = -\lambda$ and $a_{12}=-q$. Now (\ref{E2.7.2}) yields
$-\lambda g_1 = -\lambda g_1^2 - qg_1 h + hg_1$, so
\begin{eqnarray*}
  \eta_{g_1} \left( h \right) - q h  =  \lambda \left( g_1 - 1 \right) \in k G.
\end{eqnarray*}

Since $h$ is $(g_1,1)$-primitive, by Lemma \ref{lemxx2.2}(c), we have that $q\neq 1$ or $h \in kG$. In the latter case, $K \cong kG$, which yields a contradiction, so $U$ is decomposable. On the other hand, if $q\neq 1$, let $v'= v+a(1-q)^{-1}$. Then
$r=v'u-quv'-\lambda u^2$.
Since $g_2=1$, we have that $\rho(v')=v'\otimes 1+u\otimes h$.
Thus, the $K$-coaction on $R$ is graded when using the basis $\{u,v'\}$.
This contradicts the hypothesis, so again $U$ is decomposable.

(b,c) By part (a), we have an isomorphism $U
\cong k u \oplus k v$ of $K$-comodules. Thus $\rho (u) = u \otimes
g_1$, $\rho (v) = v \otimes g_2$, and by the argument after
\eqref{E2.7.1}, the elements $g_1, g_2$ are grouplike. Note that
\eqref{E2.7.2}-\eqref{E2.7.8} hold and $h=0$ in this case. Since $(a, b)\neq (0,0)$
we can assume, by symmetry, that $a \neq 0$. Then, $g = g_1$ by \eqref{E2.7.6}. Since $\gr_F R$ is a domain,
either $a_{11}$, or $a_{12}$, or $a_{21}$ is nonzero. So one of the equations
\eqref{E2.7.2}- \eqref{E2.7.4} implies that $g_1=1$ or $g_2=1$.
Since $K$-coaction on $R$ is inner-faithful, we see that $K$ is generated by a single grouplike element. So $K
\simeq k C_m$ for some $m\geq 2$.

 (d,e) Since $a, b$ are not both zero, we may assume $b\neq 0$ by
symmetry, which implies that $g=g_2$ by \eqref{E2.7.7}. Recall that $h=0$ in this case. By part (b),
we have either $g_2 = 1$ or  $g_1 = 1$. So we have two cases to
consider:

Case 1: $g=g_2=1$ and $g_1\neq 1$. By \eqref{E2.7.2}-\eqref{E2.7.8},
we have that $a_{12}=a_{21}=a=0$. So we have that $r=a_{11}
u^2+a_{22} v^2+bv+c$. Replacing $v$ with a change of variables, we
can make $a=b=0$. This contradicts the hypothesis.

Case 2: $g=g_2\neq 1$ and $g_1=1$. By \eqref{E2.7.2}-\eqref{E2.7.8},
we have $a_{11}=a_{22}=a=c=0$. Up to scaling, $r=vu-quv+b'v$ for $b' \in k^{\times}$. If
$q\neq 1$, we may replace $u$ by $u':=u+b'(1-q)^{-1}$ so that the
relation becomes $vu'-qu'v$. Since $g_1=1$, then $K$ coacts on the new
relation and whence the $K$-coaction on $R$ is graded, yielding a
contradiction. Therefore $q=1$, and the assertion in (d) follows.

Since $g_1=1$, $\rho(r)=r\otimes g_2$. This means that $\cohdet_K \gr_F R=
g_2\neq 1$. Hence, part (e) follows.
\end{proof}

\begin{remark} \label{rmk:HinSec2}
Suppose $H$ acts on the filtered algebra $R=k\langle u,v\rangle/(r+s)$ where $r\in F_2$ and $s\in F_1$.
Then by Lemmas \ref{lemxx2.6} and \ref{lemxx2.7}(e), the induced $H$-action on $\gr_F R$ has trivial
homological determinant if and only if $s\in k$. In this case, an $H$-action on $\gr_F R$ lifts uniquely to an
$H$-action of $R$, so there is a bijective correspondence
between $H$-actions on $R$ and $H$-actions on $\gr_F R$ with trivial homological determinant.
\end{remark}

Moreover, as a result of Lemma \ref{lemxx2.7}, we can classify the relations of $R$ so that $H$ acts on $R$ under Hypothesis \ref{hypxx2.3}.

\begin{corollary} \label{corxx2.8}
The relation $r$ of $R$ under Hypothesis \ref{hypxx2.3} is in one of
the following forms:
\[
\begin{array}{llll}
vu-quv-1, &vu-uv-u^2-1, & \text{or} & vu-uv -v,
\end{array}
\]
for $q \in k^{\times}$, up to a change of basis of $U$.
\end{corollary}

\begin{proof}
Apply Lemma \ref{lemxx2.6} and Lemma \ref{lemxx2.7}(d,e).
\end{proof}

Now we turn our attention to PI algebras (algebras that satisfy a
polynomial identity). The following result provides a sufficient
condition for a filtered AS regular algebra of dimension 2 to be PI.

\begin{lemma}
\label{lemxx2.9} Let $R$ be a filtered AS regular algebra of
dimension 2 with relation
$r=a_{11}u^2+a_{12}uv+a_{21}vu+au+bv+c$ for $a_{ij}$,
$a,b,c \in k$. Suppose that $a_{12} + a_{21}\neq 0$ and $\gr_F R$
is PI. Then $R$ is PI. Equivalently, if $R$ is not PI, then either $\gr_F R$ is not PI or $\gr_F R$ is commutative.
\end{lemma}

\begin{proof} Since $\gr_F R$ is a domain, we have $a_{12} a_{21} \neq 0$. So we may
assume $a_{21} = 1$ and $a_{12} = - q \neq 0$. By the first hypothesis, we
have that $q \neq 1$. By replacing $v$ by $v - a_{11} \left( 1 - q \right)^{- 1} u$
we can assume $a_{11} = 0$, so $r = v u - q u v + a u + b v + c$. Now by replacing
$u$ with $u - b \left( 1 - q \right)^{- 1}$ and $v$ with $v - a \left( 1 - q
\right)^{- 1}$, we can assume that $r = v u - q u v + c$. Since $\gr_F R$
is PI, we see that $q$ is an $n$-th root of unity for $n\geqslant 2$. Therefore $u^n$ and $v^n$ are central, so $R$ is module finite over its center. Hence, $R$ is PI.

Now if $R$ is not PI, we just showed that $\gr_F R$ is not PI or $a_{12} + a_{21}= 0$, with $a_{12},a_{21}\neq 0$. 
Thus if $\gr_F R$ is PI, it is isomorphic to $k\langle u,v\rangle/(r_{\lambda})$ where $r_{\lambda}=vu-uv+\lambda u^2$ 
for some $\lambda \in k$. For all $\lambda \neq 0$, the algebra $k\langle u,v\rangle/(r_{\lambda})$
is isomorphic to the Jordan plane, which is not PI. This shows that $\lambda =0$,
so $\gr_F R$ is commutative. 
\end{proof}

\section{Fixed subrings and the proof of Theorem \ref{thmxx0.4}}
\label{xxsec3}

In this section, we provide several results about the fixed ring
$R^H$ corresponding to a finite dimensional Hopf algebra $H$ acting on a filtered
AS-regular algebra $R$. In particular, we prove Theorem \ref{thmxx0.4} and
a weakened version of Theorem \ref{thmxx0.5} (see Proposition \ref{proxx3.3} and the material after). We end the section by computing examples of fixed rings of Hopf actions on PI algebras [Example \ref{exxx3.4}].

Given any algebra $A$ that is not necessarily filtered AS regular,
suppose that a Hopf algebra $H$ acts on $A$. Then the {\it fixed
subring} of the $H$-action is defined to be
$$A^H=\{a\in A\mid h\cdot a=\epsilon(h)a
\quad {\text{for all $h\in H$}}\}.$$  Now let $A$ be a filtered
algebra with a nonnegative exhaustive filtration $\{F_n
A\}_{n\geqslant 0}$. For any $x\in F_i A$, we use $\overline{x}$
(or $\overline{x}_i$) for the
corresponding image of $x$ in $A_i:=F_i A/F_{i-1} A$.
Suppose that $H$ acts on $A$ such that each $F_n A$ is
a left $H$-module. For each $h\in H$, and for each homogeneous element
$\overline{x}\in A_i$ where $x\in F_i A$ is any lift of
$\overline{x}$, define
$$h \cdot \overline{x}=\overline{(h\cdot x)}_i.$$
It is possible that $h\cdot x\in F_{i-1} A$, but we want to consider the
image of $h\cdot x$ in $A_i$. It is easy to check that $H$ acts on
$\gr_F A$ so that $\gr_F A$ is a left $H$-module algebra. We record this
without proof as part (a) of the following lemma.

\begin{lemma}
\label{lemxx3.1} Let $A$ be a filtered algebra with filtration
$\{F_n A\}_{n\geqslant 0}$. Suppose that $H$ acts on $A$ such that
each $F_n A$ is a left $H$-module. Then, the following statements hold.
\begin{enumerate}
\item
$H$ acts on $\gr_F A$ naturally.
\item
$A^H$ has an induced exhaustive filtration $F'$ such that $\gr_{F'} (A^H)$
is a subalgebra of $(\gr_F A)^H$.
\item
If $H$ is semisimple (whence finite dimensional), then
$\gr_{F'} (A^H)=(\gr_F A)^H$.
\end{enumerate}
\end{lemma}

\begin{proof} (b) Let $F'_n=A^H\cap F_n A$. Then $\{F'_n\}_{n\geq
0}$ is a nonnegative exhaustive filtration of the subalgebra $A^H$
of $A$. Clearly, $\gr_{F'} (A^H)$ is a subalgebra of $\gr_F A$. For
any nonzero homogeneous element $\overline{x} \in \gr_{F'} (A^H)$,
we pick a lift $x\in A^H$. By definition, $$h\cdot
\overline{x}=(\overline{h\cdot x})_i= (\overline{\epsilon(h) x})_i =
\epsilon(h)\overline{x}.$$ Hence $\overline{x} \in (\gr_F A)^H$. The
assertion follows.

(c) Consider the induced subfiltration $F'_n=A^H \cap F_n A$ of $A^H$.
By part (b), it suffices to show that $\gr_{F'} (A^H)\supseteq
(\gr_F A)^H$. Since $H$ is semisimple,  we may choose a left
integral element $\inth \in H$ with $\epsilon(\inth) =1$. Moreover, a left trace function $\tr: A
\rightarrow A^H$, defined by $a \mapsto \inth \cdot a$, is surjective as $H$ is semisimple. Hence, $A^H = \inth \cdot A$.
For every nonzero homogeneous element $\overline{x}\in (\gr_F
A)^H\subset \gr_F A$ of degree $i$,
$$\overline{x}=\inth \cdot \overline{x} =\overline{\inth \cdot x}.$$
This means that $x-\inth\cdot x$ in $F_{i-1} A$. So we may replace
$x$ by $\inth\cdot x$ and assume that $x\in A^H$. Therefore,
$\overline{x}$ is in $\gr_{F'} (A^H)$.
\end{proof}

Part (c) of the lemma above need not hold if $H$ is not semisimple;
we illustrate this in the following example. We also provide an
example of an inner-faithful $H$-action on an algebra $A$ so that
the induced action on $\gr_F A$ is not inner-faithful.

\begin{example}
\label{exxx3.2} Here, we do not assume that $k$ is of characteristic
zero. Let $H$ be Sweedler's 4-dimensional Hopf algebra $k\langle g,
f\rangle/(fg+gf, ~g^2-1, ~f^2)$. (See Example~\ref{exxx1.4} for the
coalgebra structure and antipode of $H$.) Let $A=k[u]$ and define a
left $H$-action on $A$ by
$$ f\cdot u= 1 \quad {\text{and}}\quad g\cdot u=-u.$$
It is easy to check that $A$ is a left $H$-module algebra.
The following statements are also easy to check.
\begin{enumerate}
\item
The $H$-action on $A$ is not proper.
\item
The $H$-action on $A$ is inner-faithful.
\item
By induction, we have that
$$f\cdot u^n=\begin{cases} 0,& n={\text{even}}\\ u^{n-1}, &n={\text{odd}}
\end{cases}$$
for all $n$. As a consequence, $A^H=k[u^2]\neq k[u]=A$.
\item
Let $F$ be the filtration defined by $F_n = (k + ku)^n$. Then $\gr_F A\cong k[u]$ with $\deg
u=1$.
Take $\overline{u}$ to be the image of $u \in F_1 A$ in $A_1$.
Hence, $f \cdot \overline{u} = (\overline{f \cdot u})_1 =
\overline{1}_1 =0$. Likewise for $g$, we see
that the $H$-action on $\gr_F A$ is determined by
$ f\cdot \overline{u}= 0$ and $g\cdot \overline{u}=-\overline{u}.$
As a consequence, the $H$-action on $\gr_F A$ is \underline{not}
inner-faithful.
\item
Let $F' =F \cap A^H$ be the induced filtration of $A^H$. If $\mathrm{char} \; k\neq 2$, then $(\gr_F A)^H=\gr_{F'} A^H$.
\item
If $\mathrm{char} \; k=2$, then the $H$-action on $\gr_F A$ is trivial.
So $(\gr_F A)^H=\gr_F A$. As a consequence,
$(\gr_F A)^H\not\cong \gr_{F'} A^H$.
\end{enumerate}
\end{example}

Next, we prove weakened versions of Theorems \ref{thmxx0.4} and
\ref{thmxx0.5}.

\begin{proposition}
\label{proxx3.3} Let $H$ be a semisimple Hopf algebra that acts on a
filtered AS regular algebra $R$ of dimension 2, inner-faithfully and
preserving filtration, with the $H$-action not graded. Note that the
$H$-action on $R$ is proper by Lemma \ref{lemxx1.5}(a). If the
homological determinant of the $H$-action of $\gr_F R$ is not
trivial, then $R^H$ has global dimension 2.
\end{proposition}

\begin{proof}
Since the homological determinant of the $H$-action on $\gr_F R$ is
not trivial, $(a,b)\neq (0,0)$ by Lemma \ref{lemxx2.6}. Now by Lemma
\ref{lemxx2.7}(c,d), the relation of $R$ is of the form $r=vu-uv-v$,
the Hopf algebra is $H =kC_m= k\langle h \rangle$, and the action of
$H$ on $R$ is given by
$$h\cdot u=u,\quad h\cdot v=\xi v$$
for some primitive $m$-th root of unity $\xi$. It
is easy to see that the fixed subring $R^H$ is
$$R^H = k\langle u, v^m\rangle/(v^m u-(u+m)v^m).$$
So $R^H$ has global dimension $2$, since it is isomorphic to the Ore extension
$k \left[ u \right] \left[ v^m, \sigma \right]$ where $\sigma \in \Aut
\left( k \left[ u \right] \right)$ is given by $\sigma \left( u \right) = u +
m$.
\end{proof}

Now we are ready to prove Theorem \ref{thmxx0.4}.

\begin{proof}[Proof of Theorem \ref{thmxx0.4}]
Since $H$ is semisimple, the $H$-action on $R$ is proper by Lemma
\ref{lemxx1.5}(a). If the homological determinant of the $H$-action
on $\gr_F R$ is not trivial, then the assertion follows from
Proposition \ref{proxx3.3}. If the homological determinant of the
$H$-action on $\gr_F R$ is trivial, then by \cite[Theorem 0.1]{KKZ},
$(\gr_F R)^H$ is AS Gorenstein. By  Lemma \ref{lemxx3.1}(c), $\gr_F
(R^H)\cong (\gr_F R)^H$, which is AS Gorenstein, and by definition,
$R^H$ is filtered AS Gorenstein.
\end{proof}

Furthermore, we compute two examples of fixed
subrings of Hopf actions on
PI algebras.

\begin{example}
\label{exxx3.4}
Let $R$ be the algebra $k\langle u,v\rangle/(u^2+v^2-1)$. We
will consider two different Hopf actions on $R$, and compute
the corresponding fixed subrings $R^H$. Note that $\gr_F R$ 
is AS regular, and the following actions of $H$ on $R$ are 
filtered AS regular actions.

(1) Let $H_8$ be the (unique) $8$-dimensional noncommutative and
noncocommutative semisimple Hopf algebra. It is generated as an
algebra by $x,y,z$ and subject to the relations:
$$\begin{aligned}
x^2=1,\quad  y^2=1,\quad  z^2&=\frac{1}{2}(1+x+y-xy),
&xz=zy, \quad yz=zx,\quad  xy&=yx.
\end{aligned}
$$
The coalgebra structure and antipode of $H_8$ is determined by
\begin{align*}
\triangle(x)=x\otimes x, \quad
\triangle(y)=y\otimes y, \quad
\triangle(z)=\frac{1}{2}(1\otimes 1+1\otimes x+y\otimes 1
-y\otimes x)z\otimes z,\\
\epsilon(x)=1, \quad \quad
\epsilon(y)=1, \quad \quad
\epsilon(z)=1, \quad \quad
S(x)= x, \quad \quad
S(y)= y, \quad \quad
S(z)=z.
\end{align*}
Consider the $H_8$-action on $R$ is is given by
\[
\begin{array}{lllll}
x\cdot u=-u,&& y\cdot u=u, &&z\cdot u=v,\\
x\cdot v=v,  && y\cdot v=-v, && z\cdot v=u.
\end{array}
\]
Denote
\begin{eqnarray*}
a=(uv)^2-(vu)^2, \hspace{.2in} b=u^4-u^2+\frac{1}{4}, \hspace{.2in}
c=\left(u^2-\frac{1}{2}\right) ((uv)^2+(vu)^2).
\end{eqnarray*}
It is not hard to check that the fixed subring
$R^{H_8}$ is a commutative ring which is isomorphic to
$k[a, b,c]/(c^2-b(a^2+4(b-\frac{1}{4})^2))$.

Note that $R^{H_8}$ is Gorenstein by Theorem \ref{thmxx0.4}. Moreover, $R^{H_8}$
 has isolated singularities at $(a,b,c)=(0,\frac{1}{4},0)$ and
$(\pm \sqrt{-1}/2,0,0)$. These are Kleinian singularities both of type A$_1$.

(2) Let $H=(k D_{2n})^\circ$ where $D_{2n}$ is the dihedral group of
order $2n$. Since the $H$-action is equivalent to the
$H^\circ$-coaction, we will  consider the $kD_{2n}$-coaction on $R$.

By definition, $D_{2n}=\langle x,y|x^2=y^2=(xy)^n=1\rangle$. Define a
comodule structure map $\rho: R\longrightarrow R\otimes kD_{2n}$ by
\begin{eqnarray*}
\rho(u)=u\otimes x  \quad \text{~~and~~} \quad \rho(v)=v\otimes y.
\end{eqnarray*}
By a simple calculation, the fixed subring $R^{H}=R^{co (kD_{2n})}$
is a commutative ring isomorphic to $k[a,b,c]/(bc-a^n(1-a)^n)$
generated by $a=u^2$, $b=(uv)^n$ and $c=(vu)^n$.
 By Theorem \ref{thmxx0.4}, $R^{H}$ is
Gorenstein. If $n\geqslant 2$, then $R^{H}$ has isolated
singularities at $(a,b,c)=(0,0,0)$ and $(1,0,0)$. These are also Kleinian singularities both of type A$_{n-1}$.
\end{example}

\section{The proof of Theorem \ref{thmxx0.3}}
\label{xxsec4}

This section is dedicated to the proof of Theorem \ref{thmxx0.3}.
First, we introduce some notation and we provide some preliminary results.

Let $F=\{F_n A\}_{n\geqslant 0}$ denote a filtration of $A$. The
{\it Rees ring} of $A$ with respect to $F$ is defined to be
$$\Rees_F A=\bigoplus_{n \geqslant 0} (F_n A) t^n.$$
We begin by analyzing the Rees ring of the $n$-th Weyl algebra $A_n(k)$
with respect to the standard filtration. Here, $A_n(k)$ is generated
by $u_1, \dots, u_n, v_1, \dots, v_n$ subject to relations
$$[u_i, u_j]=[v_i, v_j] = 0 \quad \text{~~and~~}
\quad [v_i, u_j] =\delta_{ij}.$$
Moreover, we refer to \cite[Definition 1.7]{CWZ} for the definition of
the Calabi-Yau property in terms of Nakayama automorphisms.

\begin{lemma}
\label{lemxx4.1}
Let $A_n(k)$ be the $n$-th Weyl algebra with the standard
filtration $F$. Then the following statements regarding
$B:=\Rees_F A_n(k)$ hold.
\begin{enumerate}
\item
$B$ is generated by $u_1, \dots, u_n, v_1, \dots, v_n, t$ subject to the
relations:
$$[u_i,u_j] = [v_i,v_j] = [u_i, t] = [v_i,t] = 0, \quad
\text{and} \quad [v_i,u_j]=\delta_{ij} t^2$$
for $1\leq i,j\leq n$. We call $\{u_1,\cdots, u_n, v_1,\cdots, v_n,t\}$
the standard basis of $B$.
\item
$B$ is a Koszul AS regular algebra of global dimension $2n+1$.
\item
$B$ is Calabi-Yau.
\item
Let $\{u^*_1,\cdots, u^*_n, v^*_1,\cdots, v^*_n,t^*\}$ be the
$k$-linear dual of the standard basis. If ${\text{char}} \; k=0$
(or if ${\text{char}} \; k>n$),
then in the Koszul dual $B^!$ of $B$, we have that $(t^*)^{2n+1}\neq 0$.
\item
If $f=t^*+\sum_{i=1}^n (a_i u^*_i+b_i v^*_i)$ is in $B^!$ for some
scalars $a_i, b_i \in k$, then $f^{2n+1}=(t^*)^{2n+1}$.
\end{enumerate}
\end{lemma}

\begin{proof} (a,b) These are well-known.

(c,d) Note that the Koszul dual $B^!$ of $B$ is generated by the
$k$-linear dual of the standard basis,  subject to the relations
\[
\begin{array}{lll}
u^*_iu^*_j+u^*_ju^*_i=0,  &v^*_iv^*_j+v^*_jv^*_i=0, & u^*_i v^*_j+v^*_ju^*_i=0,
\quad (u^*_i)^2=0,
\quad (v^*_i)^2=0\\
u^*_i t^*+t^* u^*_i=0, &v^*_i t^*+t^* v^*_i=0,
 & (t^*)^2 + \displaystyle \sum_{i=1}^n v^*_iu^*_i=0,
\end{array}
\]
for all $1 \leq i,j \leq n$. Hence, $B^!$ is isomorphic to
the exterior algebra
$$\Lambda (u^*_1,\cdots, u^*_n, v^*_1, \cdots, v^*_n, t^*)$$
as a graded vector space. In particular,
$\mathfrak{e}:=v^*_1u^*_1 v^*_2u^*_2 \cdots v^*_n u^*_n t^*$ is a nonzero
element in the highest degree of $B^!$ (degree = $2n+1$). Using the relations
$(u^*_i)^2=(v^*_i)^2=0$ and $(t^*)^2 = - \sum_{i=1}^n v^*_iu^*_i$ for
all $i$, we have that $(t^{\ast})^{2n+1}=(-1)^n n! \mathfrak{e}\neq 0$.
Hence, part (d) holds.

It is easy to check that $a b=ba$ if $a\in B^!$ has degree 1 and
$b\in B^!$ has degree $2n$. Therefore, the Nakayama automorphism of
$B^!$ is the identity; refer to \cite[Section 3]{CWZ}. By
\cite[Theorem 6.3]{BM}, the Nakayama automorphism of $B$ is the
identity. Now $B$ is Calabi-Yau by definition.

(e) This follows by a direct computation.
\end{proof}

The following lemma is clear, so we omit the proof.

\begin{lemma}
\label{lemxx4.2}
Let $H$ be a finite dimensional Hopf algebra.
Let $A$ be a filtered algebra so that $H$ acts on $A$ and each $F_n A$
is a left $H$-module. Then the following statements hold.
\begin{enumerate}
\item
There is an induced $H$-action on $\Rees_F A$ such that $\Rees_F A$
is a left $H$-module algebra with each homogeneous component
of $\Rees_F A$ being a left $H$-module.
\item
The quotient map $\Rees_F A\to \Rees_F A/(t-1)=A$ is an $H$-module
algebra homomorphism.
\item
The quotient map $\Rees_F A\to \Rees_F A/(t)=\gr_F A$ is an $H$-module
algebra homomorphism.
\item
$(\Rees_F A)^H=\Rees_{F'} A^H$ where $F'$ be the induced filtration on
$A^H$. \qed
\end{enumerate}
\end{lemma}

Since $B =\Rees_F A_n(k)$ is a left $H$-module algebra, we have that
$K= H^{\circ}$ coacts on $B^!$ from the left
\cite[Remark~1.6(d)]{CWZ}. Here, the $K$-comodule structure map is
denoted by $\rho: B^!\rightarrow K \otimes B^! $.

We now define an algebra $\tilde{B}^!$ that will aid in the study of
the $K$-coaction on $B^!$. Let $\Lambda = \Lambda \left( z_1,
\ldots, z_{2 n} \right)$ denote the exterior algebra in $2n$
variables and define $\tilde{B}^! = \Lambda \left[ y ; \sigma
\right]$, where $\sigma \left( z_i \right) = - z_i$ for all
$i=1,\cdots, 2n$. By recalling the presentation of $B^!$ from the
proof of Lemma \ref{lemxx4.1}(c), we see that there is a $k$-algebra
isomorphism
$$B^! \cong \tilde{B}^! / \left( y^2 +  \sum_{i = 1}^n z_i z_{n + i}\right)$$
given by $t^*\mapsto y$, $v_i^{\ast} \mapsto z_i$ and $u_i^{\ast}
\mapsto z_{n + i}$ for all $i=1,\cdots,n$. For convenience, we also
use $z_1,\cdots,z_{2n}$ and $y$ as the corresponding generators for
both $B^!$ and $\tilde{B}^!$.

\begin{lemma}\label{lemxx4.3} Let
$\{c_i\}_{i=1}^{2n}$ be a subset of $K$ and let $s := c_1 \otimes z_1
+ \cdots +  c_{2n} \otimes z_{2 n} $. For any $r> 0$, the following
identity holds in $K \otimes \tilde{B}^! $
$$    \left( 1 \otimes y+ s \right)^{2 r} = \sum_{j = 0}^r \binom{r}{j}
     \sum_{\sigma \in S_{2 j}} \left( - 1 \right)^{\mathrm{sgn}(\sigma)} c_{i_{\sigma \left( 1 \right)}} \cdots
    c_{i_{\sigma \left( 2 j \right)}}
    \otimes
\sum_{i_1 < \cdots < i_{2 j}} y^{2( r- j)} z_{i_1}\cdots z_{i_{2j}}
    .
$$
\end{lemma}

\begin{proof}
  Since $1\otimes y$ and $s$ skew-commute, we have $(1\otimes y + s)^2
  = 1\otimes y^2 +
  s^2$. Moreover, $1\otimes y^2$ and  $s^2$ commute, so we can use the usual binomial
  theorem
  to get
  \begin{eqnarray*}
    ( 1\otimes y + s)^{2 r} & = & \sum_{j = 0}^r
\binom{r}{j} (1 \otimes y^{2 \left( r  - j \right)} )s^{2 j}.
  \end{eqnarray*}
Since $s \in  K \otimes \Lambda$ and $\Lambda$ is the exterior
algebra, the formula follows by expanding $s^{2j}$ in the algebra
$K \otimes \Lambda $.
\end{proof}

Let $c \in K$, then left multiplication by $c$ defines a $k$-vector space
endomorphism $\mu_c : K \longrightarrow K$. We define $\tr \left( c
\right)$ to be the trace of $\mu_c$ as an element of $\End_k \left( K
\right)$.

\begin{lemma}
\label{lemxx4.4}
  Let $\rho : B^! \longrightarrow K \otimes B^!$ be a $K$-coaction of $B^!$
  with $\rho \left( y \right) = 1\otimes y + \sum_{i = 1}^{2 n} c_i\otimes
  z_i $ for $\left\{c_i\right\}_{i=1}^{2n}$ a subset of $K$, then
  \begin{eqnarray*}
    \rho \left( y \right)^{2 n + 1}  =  \left( 1 + p
    \right) \otimes y^{2 n + 1}
  \end{eqnarray*}
  where $p\in K$ with $\tr \left( p \right) = 0$.
\end{lemma}

\begin{proof} Working with algebras $\Lambda \subset B^!$, we have
  $y^2 =-\sum_{i=1}^n z_i z_{n+i} \in \Lambda_2$. By Lemma \ref{lemxx4.1}(d),
$y^{2n}\neq 0$, so take it as a basis element for $\Lambda_{2n}$. For
$1\leq i_1<i_2<\cdots <i_{2j}\leq 2n$, we have that
$y^{2 \left( n - j \right)} z_{i_1} \cdots
  z_{i_{2 j}} \in \Lambda_{2 n} = k y^{2 n}$. Write
  $$y^{2 \left( n - j \right)} z_{i_1} \cdots
  z_{i_{2 j}} =\lambda_{i_{1}\cdots i_{2j}} y^{2n}$$
  for some $\lambda_{i_1\cdots i_{2j}}\in k$.
   So by the previous lemma, we
  have that
  \begin{eqnarray*}
    \rho \left( y \right)^{2 n} & = & \sum_{j = 0}^n
   \binom{n}{j} \sum_{i_1 < \cdots < i_{2 j}}  \lambda_{i_{1} \cdots
    i_{2j}}  \left( \sum_{\sigma \in S_{2 j}}
    \left( - 1 \right)^{\sigma} c_{i_{\sigma \left( 1 \right)}} \cdots
    c_{i_{\sigma \left( 2 j \right)}} \right) \otimes  y^{2 n}
    .
  \end{eqnarray*}
  For $j=0$, the above expression is $1\otimes y^{2n}$. If $j > 1$, then
  \begin{eqnarray*}
    \tr \left( \sum_{\sigma \in S_{2 j}} \left( - 1 \right)^{\sigma}
    c_{i_{\sigma \left( 1 \right)}} \cdots c_{i_{\sigma \left( 2 j \right)}}
    \right) & = & \tr \left( \sum_{\sigma \in S_{2 j}} \left( - 1
    \right)^{\sigma} c_{i_{\sigma \left( 2 j \right)}} c_{i_{\sigma \left( 1
    \right)}} \cdots c_{i_{\sigma \left( 2 j - 1 \right)}} \right)\\
    & = & - \tr \left( \sum_{\sigma \in S_{2 j}} \left( - 1
    \right)^{\sigma} c_{i_{\sigma \left( 1 \right)}} \cdots c_{i_{\sigma
    \left( 2 j \right)}} \right).
  \end{eqnarray*}
  The above is zero since $\ch(k)\neq 2$.
  This shows that $\rho \left( y \right)^{2 n} =  \left( 1 + p
  \right)\otimes  y^{2 n}$ where $\tr \left( p \right) = 0$. Finally, note that $y^{2 n}
  z_i = z_i y^{2 n} = 0$, so
$$\rho \left( y \right)^{2 n + 1} = \left( 1
  \otimes y \right) \rho \left( y \right)^{2 n} =    \left( 1
  + p \right)\otimes y^{2 n + 1}.$$ This completes the proof.
\end{proof}

Now we are ready to prove Theorem \ref{thmxx0.3}.

\begin{proof}[Proof of Theorem \ref{thmxx0.3}]
Let $B:=\Rees_F A_n(k)$, and let $B^!$ denote the Koszul dual of
$B$. Since $B$ is Calabi-Yau [Lemma \ref{lemxx4.1}(c)], it suffices
to show that the homological determinant of the left $H$-action on
$B$ is trivial \cite[Theorem 0.6]{CWZ}. Equivalently, we show that
the left $K = H^{\circ}$-coaction on $B^!$ has trivial homological
codeterminant.

Note that since $H$ acts on $A_n(k)$ preserving the filtration,
$H$ acts on $B$ inner-faithfully. So $K$ coacts on $B$
inner-faithfully. Hence, $K$ coacts on $B^!$ inner-faithfully
\cite[Proposition 2.5(c)]{CWZ}. Let $\rho$ denote the $K$-coaction
on $B^!$. Note that $T := k t$ is a trivial $K$-sub-comodule of
$W = k t \oplus \bigoplus_{i = 1}^n (k u_i \oplus k v_i)$ by
Lemma \ref{lemxx4.2}(b). We have a $K$-comodule map $\pi :
W^{\ast} \longrightarrow T^{\ast}$ which sends
$u^{\ast}_i$ and $v_i^{\ast}$ to zero. Since $T^{\ast}$ is also a trivial
$K$-comodule, that is $\left( 1 \otimes \pi \right) \rho \left( t^{\ast}
\right) = 1 \otimes \pi \left( t^{\ast} \right)$, we have that
\begin{eqnarray*}
  \rho \left( t^{\ast} \right) & = & 1\otimes t^{\ast}  + \sum_{i = 1}^n
(a_i \otimes u_i + b_i \otimes v_i)
\end{eqnarray*}
for some $a_i, b_i \in K$. By Lemma \ref{lemxx4.3}, we have
$$\rho((t^*)^{2n+1})=(1+p)\otimes (t^*)^{2n+1}$$
where $\tr(p)=0$.
By definition, the homological codeterminant $\sf{D}$ of the
$K$-coaction on $B^!$ is $1+p$. Moreover, $\sf{D}$ is a grouplike element.
Since $K$ is finite dimensional, $\sf{D}$ also has finite order. Now
$\tr(p)=0$ implies that $\tr({\sf D})=\tr(1)=\dim K$.
Since ${\sf D}$ has finite order, we have that ${\sf D}=1$.
Therefore the $K$-coaction on $B$ has trivial homological
codeterminant. Dually, the $H$-action on $B$ has trivial
homological determinant as desired. Thus, $H = K^{\circ}$ is semisimple.

Since $\mathrm{char}~k = 0$, $H$ is also cosemisimple, that is, $H$ equals its coradical
$H_0$. If $H$ is pointed, then $H_0 = k G(H)$. Hence, $H$ is a group
algebra.
\end{proof}

\section{The proof of Theorem \ref{thmxx0.1} and consequences}
\label{xxsec5} We return to the study of Hopf algebra actions on
filtered AS regular algebras $R$ of dimension $2$. In
this section, we prove Theorem \ref{thmxx0.1}, and as a consequence,
we classify the possible Hopf algebra actions when $R$ is non-PI.

Note that if $R$ is a non-PI filtered AS regular algebra of
dimension 2, then it follows from Lemma \ref{lemxx2.9} that the
associated graded ring $\gr_F R$ is either non-PI or commutative. We
provide preliminary results for these cases separately.

We have the following setup. Let $K$ be a finite dimensional Hopf algebra
coacting on a non-PI filtered AS regular algebra $R$. Let $\rho$ denote the
coaction and $R \simeq k \langle u, v \rangle / \left( r \right)$. Since
$\rho$ preserves the filtration, we can write
\begin{align}
  \rho \left( u \right) & \quad =\quad
u \otimes e_{11} + v \otimes e_{21} + 1 \otimes
  f_1  \label{E5.3.1}\tag{E5.0.1}\\
  \rho \left( v \right) & \quad = \quad
u \otimes e_{12} + v \otimes e_{22} + 1 \otimes
  f_2 \label{E5.3.2}\tag{E5.0.2}
\end{align}
for some $e_{i j}, f_j \in K$, $i,j = 1, 2$. Using coassociativity of the coaction, we have
\begin{align}
\Delta(e_{ij})  &\quad =\quad
\sum_{l=1}^2 e_{il} \otimes e_{lj} \label{E5.3.3}\tag{E5.0.3}\\
\Delta(f_1)  &\quad = \quad f_1 \otimes e_{11} + f_2 \otimes e_{21} + 1 \otimes f_1
\label{E5.3.4}\tag{E5.0.4}\\
\Delta(f_2)  &\quad = \quad f_1 \otimes e_{12} + f_2 \otimes e_{22} + 1 \otimes f_2
\label{E5.3.5}\tag{E5.0.5}\\
\epsilon(e_{ij}) & \quad =\quad  \delta_{ij} \quad \quad  \text{and} \quad \quad \epsilon(f_i) \quad = \quad 0.
\label{E5.3.6}\tag{E5.0.6}
\end{align}

\subsection{$\gr_F R$ is non-PI}

We need the following well-known lemma.

\begin{lemma}
\label{lemxx5.3}
Suppose $G$ is a finite group acting on $A = k_J
  \left[ u, v \right]$ or $k_q \left[ u, v \right]$ for $q \neq \pm 1$.
  Then $G$
  is abelian and the action is diagonal with respect to the basis $\{u,v\}$.  \qed
\end{lemma}

\begin{lemma}
  \label{lemxx5.4} Suppose that $\gr_F R$ is non-PI. Then $H$ is a
  commutative group algebra.
\end{lemma}

\begin{proof}
  Let $K'$ be the Hopf subalgebra of $K$ generated by $\left\{ e_{i j}
\right\}_{i, j = 1}^2$. Then by definition $K'$ coacts on $\gr_F R$
inner-faithfully. Since the $K'$-coaction on $\gr_F R$ is inner-faithful, by
  \cite[Theorem 5.10]{CWZ}, we have that $K'$ is the dual of a finite group
  algebra. By Lemma \ref{lemxx5.3}, the coaction $\rho'$ is diagonal with
respect to $\{u,v\}$ where $\gr_F R=k_q[u,v]$ for $k_J[u,v]$. Hence we
can write
$$\rho \left( u \right)  =  u \otimes e_{11} + 1 \otimes f_1 \quad \text{and} \quad
    \rho \left( v \right)  =  v \otimes e_{22} + 1 \otimes f_2$$
  where $e_{i i}$ is grouplike (by Lemma~\ref{lem:1dimlcoact}) and $f_i$ is $\left( 1, e_{i i}
  \right)$-primitive for $i = 1, 2$.

  Suppose that  $\gr_F R  =k \langle u,v \rangle / \left( v u - q u v
  \right)$. Then the relation $r$ of $R$ is of the form $r = v u - q u v + a u + b v + c$. Note that $\rho \left( r
  \right) = r \otimes g$ for some grouplike element $g$. So, we have that
  \begin{eqnarray*}
    \rho \left( r \right) & = & v u \otimes e_{22} e_{11} - q u v \otimes
    e_{11} e_{22} + u \otimes \left( f_2 e_{11} - q e_{11} f_2 + a e_{11}
    \right) +\\
    &  & v \otimes \left( e_{22} f_1 - q f_1 e_{22} + b e_{22} \right) + 1
    \otimes \left( f_2 f_1 - q f_1 f_2 + a f_1 + b f_2 + c \right)
  \end{eqnarray*}
  By comparing the coefficients of $u$, we have that  $a g = f_2 e_{11} - q e_{11}
  f_2 + a e_{11}$. In particular, $\eta_{e_{11}} \left( f_2 \right)
- q f_2 \in k
  G(K)$. Since $q \not\in \mathbb{U} \cup \left\{ 1 \right\}$, by Lemma
  \ref{lemxx2.2}(b), we have that $f_2 \in k G(K)$. Similarly $f_1 \in k G(K)$, hence
  $K$ is a group algebra.

Finally, by comparing coefficients of $v u$ and $u
  v$, we get $e_{22} e_{11} = g$ and $- q g = - q e_{11} e_{22}$, so $\left[
  e_{11}, e_{22} \right] = 0$. Since $f_i \in k G(K)$ is $\left( 1, e_{i i}
  \right)$-primitive we have $f_i \in k \left( 1 - e_{i i} \right)$ for $i =
  1, 2$. Hence $K$ is commutative and generated by grouplike elements
$e_{11}$ and $e_{22}$. Therefore, $H=K^\circ$ is a commutative group algebra.

  Now suppose $\gr_F R = k \langle u,v \rangle / \left( v u - u v -
  u^2 \right)$. Then $r = v u - u v - u^2 + a u + b v + c$. We have
  \begin{eqnarray*}
    \rho \left( r \right) & = & u^2 \otimes (- e_{11}^2) + v u \otimes e_{22}
    e_{11} - u v \otimes e_{11} e_{22} \\
& & + u \otimes \left( \left[ f_2, e_{11}
    \right] - e_{11} f_1 - f_1 e_{11} + a e_{11} \right)\\
    &  & + v \otimes \left( \left[ e_{22}, f_1 \right] + b e_{22} \right) + 1
    \otimes \left( \left[ f_2, f_1 \right] - f_1^2 + a f_1 + b f_2 + c \right)
    .
  \end{eqnarray*}
  Again $\rho \left( r \right) = r \otimes g$ for some grouplike element $g$ by Lemma~\ref{lem:1dimlcoact}.
  By comparing the coefficients of $u^2$ and $v u$, we have that $e_{11}^2 = g =
  e_{22} e_{11}$. Hence $e_{11} = e_{22}$. Comparing the coefficients of $u$
  and $v$ give
  \begin{align}
    a g &\quad =\quad
\left[ f_2, e_{11} \right] - e_{11} f_1 - f_1 e_{11} + a e_{11},
    \label{E5.4.1}\tag{E5.2.1}\\
    b g & \quad = \quad \left[ e_{22}, f_1 \right] + b e_{22} .
\label{E5.4.2}\tag{E5.2.2}
  \end{align}
  Rearranging equation (\ref{E5.4.2}) gives $\eta_{e_{22}}
\left( f_1 \right) - f_1
  \in k G(K)$, so by Lemma \ref{lemxx2.2}(c) we have that $f_1 \in k G(K)$.
Now we
  can rearrange (\ref{E5.4.1}) to give $\eta_{e_{11}}
\left( f_2 \right) - f_2 \in
  k G(K)$, so similarly $f_2 \in k G(K)$. Now $f_i$ is
$(1,e_{ii})$-primitive,
  so $f_i \in k(1-e_{ii})$ for $i=1,2$. Since $e_{11}=e_{22}$,
we see that $K$ is generated
  by a single grouplike element, so $K=k C_m$. Consequently, $H$ is
a commutative group algebra.
\end{proof}

\subsection{$\gr_F R$ is commutative} Now we study the case where $\gr_F R \cong k[u,v]$.

\label{secxx5.2}

\begin{lemma}
  \label{lemxx5.5}
If $R$ is non-PI and $\gr_F R$ is isomorphic to $k[u,v]$, then $R$ is isomorphic to
either $A_1 \left( k \right)$ or $k \langle u,v \rangle /
\left( v u - u v -v  \right)$.
\end{lemma}

\begin{proof}
Since $\gr_F R \cong k \left[ u, v
  \right]$,  the relation $r$ of $R$ is of the form $r = v u - u v + a u + b v + c$. If $a = b = 0$, then $R \cong A_1 \left( k
  \right)$. If either $a$ or $b$ is nonzero, then by a change of variables,  $R \cong k \langle u,v \rangle / \left( v u - u v -v \right)$ as desired.
\end{proof}

Now suppose that $H$ acts on $R = k \langle u,v \rangle/ (vu-uv-v)$.

\begin{lemma}
\label{lemxx5.6}
 Retain the notation from the beginning of the section. Then $e_{11}=1$, $e_{12}=0$, and up to linear transformation $e_{21}=0$.
\end{lemma}

\begin{proof}
Consider the following computation:
  \begin{eqnarray*}
    \rho (r) & = & \rho (vu - uv -v)\\
    & = & u^2 \otimes (e_{12} e_{11} - e_{11} e_{12}) + u v \otimes (e_{12}
    e_{21} - e_{11} e_{22})\\
    &  & + v u \otimes (e_{22} e_{11} - e_{21} e_{12}) + v^2 \otimes (e_{22}
    e_{21} - e_{21} e_{22})\\
    &  & + u \otimes (f_2 e_{11} + e_{12} f_1 - e_{11} f_2 - f_1 e_{12} -
    e_{12})\\
    &  & + v \otimes (f_2 e_{21} + e_{22} f_1 - e_{21} f_2 - f_1 e_{22} -
    e_{22})\\
    &  & + 1 \otimes (f_2 f_1 - f_1 f_2 -  f_2).
  \end{eqnarray*}
  Since $\rho (r) = r \otimes g$ for some grouplike element $g$ by Lemma~\ref{lem:1dimlcoact}, we have that
  \[ \begin{array}{llllll}
       0 = e_{12} e_{11} - e_{11} e_{12}, &  &  &  &  & 0 = f_2 e_{11}+ e_{12} f_1
        - e_{11} f_2 - f_1 e_{12} - e_{12},\\
       g = e_{11} e_{22} - e_{12} e_{21}, &  &  &  &  & g = -f_2 e_{21} - e_{22}f_1
       + e_{21} f_2 + f_1 e_{22} + e_{22},\\
       g = e_{22} e_{11} - e_{21} e_{12}, &  &  &  &  & 0 = f_2 f_1 - f_1 f_2 -
       f_2,\\
       0 = e_{22} e_{21} - e_{21} e_{22} . &  &  &  &  &
     \end{array} \]
  Here, the four equations in the left column are given as
  \[ \left(\begin{array}{cc}
       e_{11} & e_{12}\\
       e_{21} & e_{22}
     \end{array}\right) \left(\begin{array}{cc}
       e_{22} & - e_{12}\\
       - e_{21} & e_{11}
     \end{array}\right) = \left(\begin{array}{cc}
       g & 0\\
       0 & g
     \end{array}\right) . \]
  Now,
  \begin{eqnarray*}
    \left(\begin{array}{cc}
      S (e_{11}) & S (e_{12})\\
      S (e_{21}) & S (e_{22})
    \end{array}\right) & = & \left(\begin{array}{cc}
      e_{22} g^{- 1}  & - e_{12} g^{- 1}\\
      - e_{21} g^{- 1} & e_{11} g^{- 1}
    \end{array}\right),
\end{eqnarray*}
and by Equations (\ref{E5.3.3}) - (\ref{E5.3.6}) and the antipode axiom, we have that
\begin{eqnarray*}
    S (f_1) & = & - f_1 e_{22} g^{- 1} + f_2 e_{21} g^{- 1},\\
    S (f_2) & = & f_1 e_{12} g^{- 1} - f_2 e_{11} g^{- 1} .
  \end{eqnarray*}
 Consider the seven relations above. By applying the antipode to the four equations in the left column above, and by using appropriate substitutions from the first two equations in the right column above, we obtain that
  \begin{eqnarray*}
    S^2 (f_1) & = & g (e_{11}+f_1 - 1) g^{- 1},\\
    S^2 (f_2) & = & g (e_{12}+ f_2) g^{- 1} .
  \end{eqnarray*}
  Let $\eta_g$ be the conjugation $a \mapsto g^{- 1} ag$. Since $\eta_g \circ
  S^2 (f_2) = e_{12}+ f_2$ and $\eta_g \circ S^2 (e_{12}) = e_{12}$, we have
  that $(\eta_g \circ S^2)^n (f_2) = f_2 + ne_{12}$. Since $K$ is finite
  dimensional, both $\eta_g$ and $S^2$ have finite order.
Now there exists $m \geqslant 1$ such that $(\eta_g
  \circ S^2)^m = \mathrm{Id}_K$. Thus $me_{12} = 0$ and $e_{12}=0$ as claimed.

On the other hand, we have that
$(\eta_g\circ S^{2})^n(f_1)= f_1+n(e_{11}-1)$. A similar argument shows that
$e_{11}=1$.

As a consequence, $e_{22}$ is grouplike and $e_{21}$ is
$(e_{22},1)$-primitive. Since $e_{22}e_{21}=e_{21}e_{22}$, we have
that $\eta_{e_{22}}(e_{21})-e_{21}=0$. Lemma \ref{lemxx2.2}(c)
implies that $e_{22}^{-1}e_{21}\in kG$. Therefore,
$e_{21}=c(1-e_{22})$ for some $c\in k$. Replacing $v$ by $v+cu$, we
have $e_{21}=0$.
\end{proof}

\begin{lemma}
  \label{lemxx5.7}
Let $R = k \langle u,v \rangle / \left( r \right)$ where $r = v
  u - u v -v$. Then $H$ is a commutative group algebra.
\end{lemma}

\begin{proof}
By Lemma \ref{lemxx5.6}, we may assume that $e_{12}=e_{21}=0$ and
$e_{11}=1$. Since $e_{11} = 1$, $f_1$ is a primitive element, so
$f_1 = 0$ [Lemma \ref{lemxx2.2}(a)]. Therefore we have that
 \begin{eqnarray*}
    \rho \left( r \right) & = & v u \otimes e_{22} - u v \otimes e_{22}
- v \otimes e_{22}  - 1 \otimes f_2.
  \end{eqnarray*}
 Since $\rho \left( r \right) = r \otimes g$ for some grouplike element
  $g$ by Lemma~\ref{lem:1dimlcoact}, we have $g=e_{22}$ and $f_2=0$.
Hence, $K$ is generated by $e_{22}$, which is a commutative group
algebra.
\end{proof}

\subsection{Proof of Theorem \ref{thmxx0.1}} Here, we prove Theorem \ref{thmxx0.1} and list an immediate
consequence afterward.

\begin{proof}[Proof of Theorem \ref{thmxx0.1}]
If $R$ is non-PI, then $\gr_F R$ is non-PI
  or $\gr_F R \cong k \left[ u, v \right]$. If $\gr_F R$ is
  non-PI, then the conclusion follows from Lemma \ref{lemxx5.4}. If $\gr_F R
  \cong k \left[ u, v \right]$, then by Lemma \ref{lemxx5.5}, $R$
is isomorphic to either $k \langle u,v \rangle /( v u - u v - v)$ or
$A_1(k)$. In the first case, the conclusion follows from Lemma
\ref{lemxx5.7}. We see that in each case
(where $R \not\cong A_1 \left( k \right)$), $K$ (and $H$) is a commutative
group algebra. Finally assume that $R=A_1(k)$. By Theorem \ref{thmxx0.3},
$H$ is semisimple. Therefore the $H$-action is proper
[Lemma \ref{lemxx1.5}(a)]. By Lemma \ref{lemxx1.5}(b,c), the induced
$H$-action on $\gr_F R$ is inner-faithful. Since
$\gr_F R\cong k[u,v]$, $H$ is a group algebra by
\cite[Proposition~0.7]{CWZ}.
\end{proof}

\begin{corollary}
\label{xxcor5.2}
  Let $R$ be a filtered AS regular algebra of dimension $2$. Suppose the
  $H$-action on $R$ is inner-faithful and preserves the filtration of
  $R$. If the $H$-action on $R$ is non-proper, then $R$ is PI.
\end{corollary}

\subsection{Additional consequences of Theorem \ref{thmxx0.1}}

In the rest of this section, we give more information about $H$-actions on
filtered AS regular algebras $R$ of dimension $2$ which are non-PI. Note
that Theorem \ref{thmxx0.1} does not provide any information about which
groups occur in the case where $R \cong A_1 \left( k \right)$.
Fortunately, this has been done in \cite[Proposition on page 84]{AHV}.

\begin{corollary}
\label{corxx5.8} Let $k={\mathbb C}$ and let $H$ be a finite
dimensional Hopf algebra acting on $A_1(\mathbb{C})$ inner-faithfully.
Then $H=kG$ where $G$ is a finite subgroup of $SL_2(\mathbb C)$,
which is conjugate to one of the following special subgroups:
\begin{enumerate}
\item[(1)]
a cyclic group of order $n$,
\item[(2)]
a binary dihedral group of order $4n$,
\item[(3)]
a binary tetrahedral group of order $24$,
\item[(4)]
a binary octahedral group of order $48$, or
\item[(5)]
a binary icosahedral group of order $120$.  \qed
\end{enumerate}
\end{corollary}

The following result classifies the $H$-module algebra structures
on filtered AS-regular algebras of dimension $2$ which are non-PI.
The following is well-known after we have shown that $H$ is a group
algebra.

\begin{corollary}
\label{corxx5.9} Let $R$ be a non-PI filtered AS regular algebra of
dimension 2 and let $H$ be a finite dimensional Hopf algebra acting
on $R$ inner-faithfully and preserving the filtration of $R$.
Suppose that the $H$-action is not graded and that $R\not\cong
A_1(k)$. Then one of the following occurs.
\begin{enumerate}
\item
$R\cong k\langle u,v\rangle /(vu-quv-1)$ for $q\in k^\times$ not a root of
unity, and $H=kG$ where $G=C_m$
with a generator $\sigma \in \Aut(R)$ determined by
$$\sigma(u)=\xi u, \quad \sigma(v)=\xi^{-1} v$$
for some primitive $m$-th root of unity $\xi$.
\item
$R\cong k\langle u,v\rangle /(vu-uv-v)$ and $H=kG$ where $G=C_m$. Up
to a change of basis, a generator $\sigma \in \Aut(R)$ is determined
by
$$\sigma(u)=u, \quad \sigma(v)=\xi v$$
for some primitive $m$-th root of unity $\xi$.
\item
$R\cong k\langle u,v\rangle /(vu-uv-u^2-1)$ and $H=kC_2$
with a generator $\sigma \in \Aut(R)$ determined by
$$\sigma(u)=-u, \quad \sigma(v)=-v.$$
\end{enumerate}
\end{corollary}

\begin{proof} By Theorem \ref{thmxx0.1},
$H=kG$ for some finite group $G$. As a consequence, the $H$-action
is proper by Lemma \ref{lemxx1.5}(a). By hypothesis, the $H$-action
on $R$ is not graded. By Corollary \ref{corxx2.8}, $R\cong k\langle
u,v\rangle/(r)$ where $r$ has the form:
\[
\begin{array}{ll}
\text{(a)} &r=vu-quv-1,\\
\text{(b)} &r=vu-uv-v, \text{~or}\\
\text{(c)} &r=vu-uv-u^2-1.
\end{array}
\]

Case (a): Since $R$ is not PI, $q$ is either 1 or not a root of
unity by Lemma \ref{lemxx2.9}. If $q=1$, $R\cong A_1(k)$, which is
excluded by hypothesis. Then, $q$ is not a root of unity. It is easy
to check that every filtered algebra automorphism $\sigma$ of $R$ is
of the form
$$\sigma(u)=\xi u, \quad \sigma(v)=\xi^{-1} v$$
for some primitive $m$-th root of unity $\xi$.

Case (b): See the proof of Lemma \ref{lemxx5.7}.
Case (c): The assertion can be proved similarly and is omitted.
\end{proof}

\section{Galois extensions and the proof of Theorem \ref{thmxx0.5}}
\label{xxsec6}
The goal of this section is to prove Theorem \ref{thmxx0.5} via the
use of Galois extensions.

\begin{definition}\cite[Definition 1.1]{CFM}
\label{defxx6.1}
Let $K$ be a Hopf algebra and $A$ be a right $K$-comodule algebra
with structure $\rho: A \rightarrow A \otimes K$.
Let $B=A^{co K}$. We say that $B\subset A$ is a (right) $K$-Galois
extension if the map $\beta: A\otimes_B A\to A\otimes K$ given by
$$\beta(a\otimes b)=(a\otimes 1)\rho(b)$$
is surjective.
\end{definition}

The following lemmas are well-known. Lemma \ref{lemxx6.4} is a
consequence of Lemma~\ref{lemxx6.2}. We use the convention that
$K:=H^\circ$, for $H$ a finite dimensional Hopf algebra.

\begin{lemma} \cite[Theorems 1.2 and 2.2]{CFM}
\label{lemxx6.2}
Let $H$ be a finite dimensional Hopf algebra and
$A$ a left $H$-module algebra. Then the following statements are equivalent.
\begin{enumerate}
\item
$A^H\subset A$ is right $K$-Galois.
\item
The map $A\# H\to \End(A_{A^H})$ is an algebra isomorphism and
$A$ is a finitely generated projective right $A^H$-module.
\item
$A$ is a left $A\# H$-generator.
\end{enumerate}
Suppose $A^H\subset A$ is right $K$-Galois. Then the following statements
are equivalent.
\begin{enumerate}
\item[(d)]
For any nonzero left integral element $t$, the corresponding
trace function  $\hat{t}: A\to A^H$ is surjective. (This holds if
$H$ is semisimple.)
\item[(e)]
$A$ is a generator for the category of right $A^H$-modules.
\item[(f)]
$A$ is a finitely generated projective left $A\#H$-module.
\end{enumerate}
In this case, $A^H$ and $A\# H$ are Morita equivalent. \qed
\end{lemma}

In the case that $H$ is semisimple, if $A^H\subset A$ is $K$-Galois,
then $A^H$ is Morita equivalent to $A\# H$. Also, $\gldim A\# H=\gldim A$
since $H$ is semisimple. We have the following remark.

\begin{remark} \label{remxx7.3}
 When  $H$ is semisimple and when $A^H$ is Morita equivalent to $A\# H$, we have that
$$\gldim A^H=\gldim A\# H=\gldim A.$$
\end{remark}

On the other hand, if $K$ is a group algebra $kG$, then $A$ is an
$K$-comodule algebra if and only if $A$ is a $G$-graded algebra.
Consider the following result.

\begin{lemma}
\label{lemxx6.4} \cite{Ul}
Let $H=kG$ and $A$ be a right $H$-comodule algebra. Let $u$ be the
identity of $G$. Then $A_u\subset A$ is $kG$-Galois if and only if
$A$ is strongly $G$-graded, if and only if $A_g A_{g^{-1}}=A_u$ for
all $g\in G$. \qed
\end{lemma}

Now we break the proof of Theorem  \ref{thmxx0.5} into two cases:
when $\gr_F R \cong k[u,v]$ and when $\gr_F R \not \cong k[u,v]$.
The first case is handled in the proposition below.

\begin{proposition}
\label{proxx6.5} Assume Hypothesis \ref{hypxx2.3} and suppose that
$\gr_F R\cong k[u,v]$. Then, up to isomorphism, $(H,R)$ occurs as
one of the following.
\begin{enumerate}
\item
$H = kC_m$, a cyclic group algebra, and
 $R\cong k\langle u,v\rangle/(vu-uv-v)$. Moreover, $R^H\subset R$ is not
$K$-Galois and  the global dimension of $R^H$ is 2.
\item \cite{AHV}
$H = kG$ where $G$ is a finite subgroup of $SL_2(k)$ and
$R=A_1(k)=k\langle u,v\rangle/(vu-uv-1)$. Also, $R^H\subset R$ is
$K$-Galois and $R^H$ is simple of global dimension 1.
\end{enumerate}
\end{proposition}

\begin{proof} By Corollary \ref{corxx2.8}, the relation is of the form
$r=vu-uv-1$ or $r=vu-uv-v$.

(a) If $r=vu-uv-v$, then by Lemma  \ref{lemxx2.7}(c) we have $H\cong k C_m$ and by
Lemma~\ref{lemxx2.7}(e), the homological determinant of the $H$-action on $\gr_F R$ is non-trivial. By
Proposition~\ref{proxx3.3}, $R^H$ has global dimension 2.

Let $\deg v=1\in C_m$ and $\deg u=0$ and write $R=\oplus_{s=0}^{m-1}
R_s$ as an $C_m$-graded algebra with respect to the degree defined
as above. It is easy to check that $1\not\in R_{1} R_{m-1}$. Thus
$R$ is not strongly $C_m$-graded, so $R^H=R_0\subset R$ is not
$K$-Galois by Lemma \ref{lemxx6.4}.

(b) If $r=vu-uv-1$, then $R$ is the Weyl algebra $A_1(k)$. By
Theorem \ref{thmxx0.1}, $H$ is a group algebra $kG$ for a finite
group $G$. By Lemma \ref{lemxx2.6}, the homological determinant of
the $H$-action on $\gr_F R$ is trivial. This means that $G\subset
SL_2(k)$. Then classical results imply that $A_1(k)^G$ is simple and
has global dimension 1 \cite[page 83]{AHV}. Since $A_1(k)$ is simple
and since $G$ is finite and does not contain any non-trivial inner
automorphisms, it is well known that $A_1(k) \# G \cong
\End(A_1(k)_{A_1(k)^G})$.  Moreover as Lemma \ref{lemxx6.2}(b)
holds, we have that $A_1(k)^G\subset A_1(k)$ is $(kG)^\circ$-Galois
by Lemma~\ref{lemxx6.2}(a).
\end{proof}

\begin{proposition}
\label{proxx6.6} Assume Hypothesis \ref{hypxx2.3} and suppose that
$\gr_F R\not\cong k[u,v]$ and that $R$ is non-PI. Then, up to
isomorphisms, $(H,R)$ occurs as one of the following.
\begin{enumerate}
\item
$H=kC_2$ and $R\cong k\langle u,v\rangle/(vu-uv-u^2-1)$,
and $R^H\subset R$ is
$K$-Galois. The global dimension of $R^H$ is 2.
\item
$H=kC_m$ and $R\cong k\langle u,v\rangle/(vu-quv-1)$ where $q$
is not a root of unity.
Here, $R^H\subset R$ is $K$-Galois and the global dimension of
$R^H$ is 2.
\end{enumerate}
\end{proposition}

\begin{proof} Similar to the beginning of the proof of the last
proposition, we may assume that $H$ is a group algebra $kG$ and that
the relation is of the form $r=vu-uv-u^2-1$ or $r=vu-quv-1$
[Corollary \ref{corxx2.8}].

(a) In this case $\gr_F R=k_J[u,v]$. The induced $H$-action on
$\gr_F R$ is inner-faithful by Lemma \ref{lemxx1.5}(a,c). To avoid
the trivial case, we assume that $H\neq k$. Now $H = k C_2=k \langle
\sigma \rangle$ by Corollary \ref{corxx5.9}(c). Furthermore, the
$\sigma$-action on $U$ is given by $\sigma(u)= -u$ and $\sigma(v)=
-v$.

Now $K\cong kC_2$ and write $R=R_{\sigma+}\oplus R_{\sigma-}$ where
$R_{\sigma\pm}=\{f\in R\mid \sigma(f)\pm f=0\}$. It is easy to see
that $1\in R_{\sigma-}^2$. Hence $R$ is strongly $C_2$-graded. By
Lemmas \ref{lemxx6.2} and~\ref{lemxx6.4}, $R^H\subset R$ is
$H^{\circ}$-Galois and $R^H$ is Morita equivalent to $R\# H$.
Therefore $R^H$ has global dimension 2 by Remark~\ref{remxx7.3}.

(b) The remaining case is when $\gr_F R=k_q[u,v]$ where $q$ is not a
root of unity. Similarly, we may assume that $H$-action on $U=ku \oplus kv$ and
hence on $\gr_F R$ are inner-faithful. Since the homological
determinant of the $H$-action is trivial, we have that $r=vu-quv-1$. By
Corollary \ref{corxx5.9}(a), $H=k C_m$ with a generator $\sigma\in
C_m$ where $\sigma(u)=\xi u$, $\sigma(v)=\xi^{-1}v$ and $\xi$ is a
primitive $m$-th root of unity. The dual Hopf algebra $K(:=H^\circ)$
is also isomorphic to $k C_m$ with a generator $\tau$ such that
$\rho(u)=u\otimes \tau$ and $\rho(v)=v\otimes \tau^{-1}$. The
assertion is verified by the following lemma.
\end{proof}

\begin{lemma}
\label{lemxx6.7} Let $R=k\langle u,v\rangle/(vu-quv-1)$ and let
$C_m=\langle\tau\rangle$ coact on $R$ by $\rho(u)=u\otimes \tau$ and
$\rho(v)=v\otimes \tau^{-1}$. If the order of $q$ is at least $m$,
then $R^{co C_m}\subset R$ is a Galois extension and $R^{co C_m}$
has global dimension 2.
\end{lemma}

\begin{proof} Write $R=\bigoplus_{s=0}^{m-1} R_s$ be the
$C_m$-graded decomposition where
$$R_s=\{f\in R \mid \rho(f)=f\otimes \tau^s\}.$$
In particular, $R_0=R^{co C_m}$. An easy computation shows that
$R_0$ is generated by $a:=u^m$, $b:=v^m$ and $c:=uv$.

Note that $R_1$ is generated by $u$ and $v^{m-1}$, and  that
$R_{m-1}$ is generated by $u^{m-1}$ and $v$. Thus $R_1R_{m-1}$
contains elements
$$u^m,\; v^m, \; v^{m-1}u^{m-1}, \quad {\text{and}}\quad uv=:c.$$
Using the relation $vu=quv+1$, we obtain that, for each $m>
s\geqslant 1$, $v^{s}u^{s}=f_{s}(c)$ for some polynomial
$f_{s}(t)\in k[t]$ of degree $s$. Moreover, $cu=u(qc+1)$ implies that $c^n
u=u(qc+1)^n$, so
$$\begin{aligned}
f_s(c)&= v^s u^s=v (v^{s-1}u^{s-1}) u\\
&= v f_{s-1} (c) u=vu f_{s-1}(qc+1)\\
&=(qc+1)f_{s-1}(qc+1).
\end{aligned}
$$
By induction, we have that
$$f_s(c)=\prod_{i=1}^s (q^ic+[i]_q)$$
where $[i]_q=1+q+\cdots+ q^{i-1}$. If the order of $q$ is at least
$m$, then $f_{m-1}(0)\in k^{\times}$. Recall that $c=uv \in R_1 R_{m-1}$. Since $f_{m-1}(0)=f_{m-1}(c)-c
(g(c))\in R_1 R_{m-1}$ for some $g(c)\in R_0$, we have that $1\in
R_1 R_{m-1}$ and $R_0=R_1 R_{m-1}$.

For any $l\geqslant 1$, we have by induction that
$R_{l}R_{m-l}\supseteq (R_1)^l (R_{m-1})^l=R_0$. This shows that $R$
is strongly $C_m$-graded. By Lemmas \ref{lemxx6.2} and
\ref{lemxx6.4}, $R^{co C_m}\subset R$ is $C_m$-Galois and $R^{co
C_m}$ is Morita equivalent to $R\# (kC_m)$. As a consequence, $R^{co
C_m}$ has global dimension two by Remark \ref{remxx7.3}.
\end{proof}

We are now ready to prove Theorem \ref{thmxx0.5}.

\begin{proof}[Proof of  Theorem \ref{thmxx0.5}]
By Theorem \ref{thmxx0.1}, $H$ is semisimple. Hence the $H$-action is
proper [Lemma \ref{lemxx1.5}(a)]. Therefore Hypothesis
\ref{hypxx2.3} holds. If $\gr_F R\cong k[u,v]$, the assertion follows from Proposition
\ref{proxx6.5}. If $\gr_F R\not\cong k[u,v]$, then the result
follows from Proposition \ref{proxx6.6}.
\end{proof}

\subsection*{Acknowledgments}
The authors thank the referee for providing many insightful
suggestions that improved the exposition of this work. C. Walton and
J.J. Zhang were supported by the U.S. National Science Foundation:
Grants DMS-1102548 and DMS-0855743, respectively. Y.H. Wang was
supported by the Natural Science Foundation of China: Grant
\#10901098 and \#11271239.

\providecommand{\bysame}{\leavevmode\hbox
to3em{\hrulefill}\thinspace}
\providecommand{\MR}{\relax\ifhmode\unskip\space\fi MR
}
\providecommand{\MRhref}[2]{%

\href{http://www.ams.org/mathscinet-getitem?mr=#1}{#2}
}
\providecommand{\href}[2]{#2}

\end{document}